\documentclass[12pt]{article}

\usepackage{amsmath,amsthm,amsfonts,amssymb,amscd,cite}
\usepackage{graphicx}

\allowdisplaybreaks[4]

\newtheorem {lemma}{Lemma}[section]
\newtheorem {theorem} {Theorem}[section]

\newtheorem {corollary}{Corollary}[section]

\allowdisplaybreaks [4]
\usepackage{fullpage}

\begin{document}

\title{\bf Normalized Laplacian eigenvalues of hypergraphs}

\author{Leyou Xu\footnote{Email: leyouxu@m.scnu.edu.cn},
Bo Zhou\footnote{Email: zhoubo@m.scnu.edu.cn}\\
School of Mathematical Sciences, South China Normal University\\
Guangzhou 510631, P.R. China}

\date{}
\maketitle

\begin{abstract}
In this paper, we give tight bounds for the normalized Laplacian eigenvalues of hypergraphs that are not necessarily uniform, and
provide an edge version interlacing theorem,  a Cheeger inequality, and a discrepancy inequality  that are related to the normalized Laplacian eigenvalues for uniform hypergraphs. 
{\bf Mathematics Subject Classifcation:} 05C65, 05C69, 05C50
\end{abstract}

\section{Introduction}\label{sec1}

Let $H$ be a hypergraph with vertex set $V(H)$ and edge set $E(H)$, where every edge is a subset of $V(H)$ containing at least two vertices.
If each edge consists of $k$ vertices, we say $H$ is $k$-uniform.
For $v\in V(H)$, denote by $E_H(v)$ ($E_v$ for convenience) the set of edges of $H$ containing $v$, i.e., $E_H(v)=\{e\in E(H):v\in e\}$.
Let $d_H(v)=|E_H(v)|$ ($d_v$ for convenience) be the degree of $v$ in $H$.
Let  $\{u,v\}\subseteq V(H)$. If there exists an edge containing both $u$ and $v$, then we say $u$ and $v$ are adjacent and otherwise we say $u$ and $v$ are non-adjacent.
Denote by $E_H(u,v)$ ($E_{uv}$ for convenience) the set of edges containing both $u$ and $v$ in $H$, i.e., $E_H(u,v)=\{e\in E(H):\{u,v\}\subseteq e\}$.
Let $e_H(u,v)=|E_H(u,v)|$ ($e_{uv}$ for convenience) be the co-degree of $u$ and $v$ in $H$.

For any nonempty proper vertex subset $S$ of a hypergraph $H$, the volume of $S$ is defined as
\[
\mathrm{vol}(S)=\sum_{v\in S}d_H(v),
\]
and the edge boundary $\partial S$ of $S$ is defined as the multiset of edges of $H$ with at least one vertex in $S$ and at least one vertex outside $S$, where the multiplicity of an edge $e\in \partial S$ is
$|e\cap S||e\cap (V(H)\setminus S)|$. It is easy to see that $|\partial S|=\sum_{u\in S}\sum_{v\in V(H)\setminus S}e_H(u,v)$.
The Cheeger constant (isoperimetric number) $h(H)$ of a hypergraph $H$ is defined as \cite{Mu}
\[
h(H)=\min\left\{\frac{|\partial S|}{\mathrm{vol}(S)}:\emptyset\ne S\subset V(H),\mathrm{vol}(S)\le \mathrm{vol}(V(H)\setminus S) \right\}.
\]

For a hypergraph $H$, Banerjee \cite{Ban} proposed  a new version adjacency matrix $A(H)=(a_{uv})_{u,v\in V(H)}$ of $H$ (with a fixed ordering of the vertices), where
\[
a_{uv}=\begin{cases}
\sum_{e\in E_{uv}}\frac{1}{|e|-1}, & \mbox{if $u$ and $v$ are adjacent},\\
0, & \mbox{otherwise.}
\end{cases}
\]
We mention that another closely related version of adjacency matrix of $H$ was proposed in \cite{FL} as the $n\times n$ matrix $B(H)=(b_{uv})_{u,v\in V(H)}$ where
\begin{eqnarray*}
b_{uv}=\begin{cases}
e_{uv}, & \mbox{if $u\neq v$, and $u,v\in V(H)$},\\[1mm]
0, & \mbox{if $u=v\in V(H)$}.
\end{cases}
\end{eqnarray*}
If $H$ is $k$-uniform, then $B(H)=(k-1)A(H)$. So there is no essential difference between  the spectral properties based on the two versions of adjacency matrix of a $k$-uniform hypergraph.
Note that for any $u\in V(H)$,
\[
\sum_{v\in V(H)\setminus\{u\}}a_{uv}=\sum_{v\in V(H)\setminus\{u\}}\sum_{e\in E_{uv}}\frac{1}{|e|-1}=\sum_{e\in E_u}\sum_{v\in e\setminus\{u\}}\frac{1}{|e|-1}=\sum_{e\in E_u}1=d_u.
\]
Let $D(H)=\mbox{diag}\{d_H(u): u\in V(H)\}$ be the degree diagonal matrix of  $H$ (with the same fixed ordering of the vertices as above).
The Laplacian matrix of $H$ is the matrix $L(H)=D(H)-A(H)$, see \cite{Ban,SSP}.
%The Laplacian eigenvalues of $H$ are just the eigenvalues of $L(H)$.
The normalized Laplacian matrix of  $H$ is defined as
\[
\mathcal{L}(H)=D(H)^{-1/2}L(H)D(H)^{-1/2}
\]
with the convention that the $(u,u)$-entry of $D(H)^{-1}$ is equal to $0$  if $d_H(u)=0$.
Clearly, this is a generalization of the normalized Laplacian matrix of a ordinary graph.
Note that the matrix $\mathcal{L}(H)$ is symmetric. So it has $|V(H)|$ real eigenvalues.
Denote by $\lambda_1(H)\ge\dots\ge \lambda_{|V(H)|}(H)$ the eigenvalues of $\mathcal{L}(H)$, which we call the normalized Laplacian eigenvalues of $H$.

The normalized normalized Laplacian eigenvalues  of ordinary graphs have been extensively, see \cite{chung}.

An edge version of Cauchy's interlacing theorem on the normalized Laplacian eigenvalues for an ordinary graph has been given in \cite{chen}. More work in this direction may be found in \cite{Bu,HJ}.

Cheeger inequality is the fundamental result in spectral graph theory, which relates
the eigenvalues of an associated matrix with combinatorial property of a graph via
 the Cheeger constant.
As the discrete analogues  of the continuous counterpart, there are
such inequalities for different versions of the Laplacian matrix (operator)
and various isoperimetric variations of  the Cheeger constant \cite{AM,Ch,chung,chung2,Mo,SSP}.
For a hypergraph $H$ with $\emptyset\ne S\subset V(H)$, $E(S, V(H)\setminus S)$ denotes the set of edges of $H$ with at least one vertex in $S$ and at least one vertex outside $S$.
By defining another version of the Cheeger constant
as
$\min\left\{\frac{|E(S, V(H)\setminus S)|}{|S|}: \emptyset\ne S\subset V(H), |S|\le |V(H)\setminus S|\right\}$,
Banerjee \cite{Ban} established the (Cheeger) inequality relating the second smallest Laplacian eigenvalue with this Cheeger constant. Mulas \cite{Mu} gave a Cheeger inequality via a different Laplacian matrix for hypergraphs. Work in this line was  studied for oriented hypergraphs by Mulas  and Zhang \cite{MZ}.

In the literatures, there are  different ways to measure the discrepancy of a graph or hypergraph \cite{BC,ch1991}. Chung \cite{ch1991} considered the volume of  each vertex subset, and take the `expectation' to be proportional to the product of the volumes. In this way, the discrepancy can be bounded using the spectral gap of the normalized Laplacian matrix of a graph. In \cite{Ch, chung2},
the discrepancy inequality bounds the edge distribution between two sets based on the spectral gap.
This was also extended to singular values of the adjacency matrix in \cite{BN} for variations of discrepancy.

In this paper, we address the normalized Laplacian eigenvalues of hypergraphs.
The rest of the paper is organized as follows. The preliminaries are provided in Section \ref{sec2}.
 In Section \ref{sec3}, we establish tight bounds for the largest and the second smallest normalized Laplacian eigenvalues of hypergraphs that are not necessarily uniform.
In Section \ref{sec4}, we prove an edge version of Cauchy's interlacing theorem  on the normalized Laplacian eigenvalues of $k$-uniform hypergraphs.
In Section \ref{sec5}, we  present the Cheeger inequality that relates the second smallest normalized Laplacian eigenvalue and  the Cheeger constant for  $k$-uniform hypergraphs.
In Section \ref{sec6}, we gave a discrepancy inequality   for $k$-uniform hypergraphs that bounds the edge distribution between two sets based on the normalized Laplacian spectral gap.
The new results exhibit some significant differences between the normalized Laplacian eihenvalues of graphs and hypergraphs.

\section{Preliminaries} \label{sec2}

If a hypergraph $H$ contains isolated vertices, then the row and column of $\mathcal{L}(H)$ corresponding to an isolated vertex in $H$ are both zero.
Let $H'$ be the hypergraph obtained from $H$ by deleting all isolated vertices.
Then each eigenvalue of $\mathcal{L}(H')$ is also an eigenvalue of $L(H)$, and $0$ is an eigenvalue of $\mathcal{L}(H)$ with multiplicity $s$ more than that of $\mathcal{L}(H')$, where $s$ is the number of isolated vertices in $H$.
So we only consider hypergraphs without isolated vertices in this paper  unless otherwise stated.
Let $\mathbf{x}$ be a nonzero $n$-dimensional column vector indexed by the vertices of a hypergraph $H$ on $n$ vertices. Then
\begin{align*}
\mathbf{x}^\top L(H)\mathbf{x}
&=\sum_{u\in V(H)}x_u\left(d_ux_u-\sum_{v\ne u}\sum_{e\in E_{uv}}\frac{1}{|e|-1}x_v\right)\\
&=\sum_{u\in V(H)}x_u\sum_{v\in V(H)\setminus\{u\}}\sum_{e\in E_{uv}}\frac{1}{|e|-1}\left(x_u-x_v\right)\\
%&=\sum_{u\in V(H)}x_u\sum_{e\in E_{u}}\frac{1}{|e|-1}\sum_{v\in e,v\ne u}\left(x_u-x_v\right)\\
&=\sum_{e\in E(H)}\frac{1}{|e|-1}\sum_{u\in e}x_u\sum_{v\in e\setminus\{u\}}(x_u-x_v)\\
&=\sum_{e\in E(H)}\frac{1}{|e|-1}\sum_{\{u,v\}\subseteq e}\left(x_u-x_v\right)^2,
\end{align*}
so
\begin{align*}
\mathbf{x}^\top \mathcal{L}(H)\mathbf{x}
=\sum_{e\in E(H)}\frac{1}{|e|-1}\sum_{\{u,v\}\subseteq e}\left(\frac{x_u}{\sqrt{d_u}}-\frac{x_v}{\sqrt{d_v}}\right)^2.
\end{align*}

For vertices $u$ and $v$ of a hypergraph $H$, a path from $u$ to $v$ in $H$ is an
alternating sequence of distinct vertices and edges of the form $v_1, e_1, v_2, e_2,\dots ,v_k$
such that $\{v_i, v_{i+1}\}\subseteq e_i$ for all $1\le i\le k-1$, where $v_1=u$ and $v_k=v$.
 If there is a path from $u$ to $v$ for any $u,v\in V(H)$, then we say that $H$ is connected.

Denote by $\mathbf{1}_n$ and $\mathbf{0}_n$ the $n$-dimensional all ones column vector and the zero vector of order $n$, respectively.

\begin{theorem}\label{n0}
If $H$ is a hypergraph on $n$ vertices, then $\mathcal{L}(H)$ is semi-definite, $\lambda_n(H)=0$ with a corresponding eigenvector $D(H)^{1/2}\mathbf{1}_n$,  and the multiplicity
of $\lambda_n(H)=0$ is exactly the number of components of $H$.
\end{theorem}

\begin{proof}
Let $\mathbf{x}$ be any nonzero $n$-dimensional column vector indexed by the vertices of $H$.
Then
\[
\mathbf{x}^\top \mathcal{L}(H)\mathbf{x}
=\sum_{e\in E(H)}\frac{1}{|e|-1}\sum_{\{u,v\}\subseteq e}\left(\frac{x_u}{\sqrt{d_u}}-\frac{x_v}{\sqrt{d_v}}\right)^2\ge 0.
\]
So $\mathcal{L}(H)$ is semi-definite.

Note that  $\mathcal{L}(H)D(H)^{1/2}\mathbf{1}_n=D(H)^{-1/2}L(H)\mathbf{1}_n=D(H)^{-1/2}\mathbf{0}_n=\mathbf{0}_n$.
So $\lambda_n(H)=0$ with a corresponding eigenvector $D(H)^{1/2}\mathbf{1}_n$.

Furthermore, suppose that $\mathbf{x}$ is an eigenvector associated with $\lambda_n(H)=0$. That is,
$\mathcal{L}(H)\mathbf{x}=\mathbf{0}_n$, equivalently, $\mathbf{x}^\top \mathcal{L}(H)\mathbf{x}=0$ as $\mathcal{L}(H)$ is semi-definite positive. Let $\mathbf{y}=D(H)^{-1/2}\mathbf{x}$. Then
$y_u=y_v$ for two adjacent  vertices $u$ and $v$ of $H$.
Any two different vertices $u'$ and $v'$ in a component of $H$ are connected by a path, so $y_{u'}=y_{v'}$.
Denote by $c$ the number of components of $H$. Then the system $\mathcal{L}(H)\mathbf{x}$ has exactly  $c$ basic solutions, one for each component of $H$.
This shows that the multiplicity of $\lambda_n(H)=0$ is exactly $c$.
\end{proof}

Given an $n$-vertex hypergraph $H$, as the trace of $\mathcal{L}(H)$ is the number of non-isolated vertices in $H$ and $H$ contains no isolated vertices, i.e.,
$\sum_{i=1}^n\lambda_i(H)= n$. %with equality if and only if $H$ has no trivial vertices.
As $\lambda_n(H)=0$, one has $\lambda_{n-1}(H)\le \frac{n}{n-1}\le \lambda_1(H)$.

Let $H$ be a hypergraph of order $n$.
Let  $\mathbf{y}=D(H)^{-1/2}\mathbf{x}\in \mathbb{R}^n\setminus\{\mathbf{0}_n\}$. Then
\begin{equation}\label{xz24}
\frac{\mathbf{x}^\top \mathcal{L}(H)\mathbf{x}}{\mathbf{x}^\top\mathbf{x}}=\frac{\sum_{e\in E(H)}\frac{1}{|e|-1}\sum_{\{u,v\}\subseteq e}(y_u-y_v)^2}{\sum_{u\in V(H)}d_uy_u^2}.
\end{equation}
We denote by $R_H(\mathbf{y})$ the expression in the right hand side of \eqref{xz24}.

\begin{lemma}\label{n-}
Let $H$ be an  $n$-vertex  $k$-uniform hypergraph.
Then
\[
\lambda_1(H)=\frac{1}{k-1}\max_{\mathbf{y}\in \mathbb{R}^n\setminus\{\mathbf{0}_n\}}\frac{\sum_{\{u,v\}\subseteq V(H)}e_{uv}(y_u-y_v)^2}{\sum_{u\in V(H)}d_uy_u^2}
\]
and
\[
\lambda_{n-1}(H)=\frac{1}{k-1}\min_{\mathbf{y}\in \mathbb{R}^n\setminus\{\mathbf{0}_n\}\atop\mathbf{y}\bot D(H)\mathbf{1}_n}\frac{\sum_{\{u,v\}\subseteq V(H)}e_{uv}(y_u-y_v)^2}{\sum_{u\in V(H)}d_uy_u^2}.
\]
\end{lemma}

\begin{proof} Let $\mathbf{y}=D(H)^{-1/2}\mathbf{x}$ for $\mathbf{x}\in \mathbb{R}^n\setminus\{\mathbf{0}_n\}$. By the %Courant-Fischer theorem
Rayleigh-Ritz theorem (see Theorem 4.2.2 in \cite{HJ}),
\[
\lambda_1(H)=\max_{\mathbf{x}\in \mathbb{R}^n\setminus\{\mathbf{0}_n\}}\frac{\mathbf{x}^\top \mathcal{L}(H)\mathbf{x}}{\mathbf{x}^\top\mathbf{x}}=\max_{\mathbf{y}\in \mathbb{R}^n\setminus\{\mathbf{0}_n\}}R_H(\mathbf{y}).
\]

By Theorem \ref{n0}, $0$ is the smallest normalized Laplacian eigenvalue of $H$
with a corresponding eigenvector $D(H)^{1/2}\mathbf{1}_n$.  Then $\mathbf{x}\bot D(H)^{1/2}\mathbf{1}_n$ if and only if $\mathbf{y}\bot D(H)\mathbf{1}_n$.
By the  Courant-Fischer theorem (see Theorem 4.2.11 in \cite{HJ}),
\[
\lambda_{n-1}(H)
=\min_{\mathbf{x}\in \mathbb{R}^n\setminus\{\mathbf{0}_n\} \atop\mathbf{x}\bot D(H)^{1/2}\mathbf{1}_n}\frac{\mathbf{x}^\top \mathcal{L}(H)\mathbf{x}}{\mathbf{x}^\top\mathbf{x}}=\min_{\mathbf{y}\in \mathbb{R}^n\setminus\{\mathbf{0}_n\}\atop\mathbf{y}\bot D(H)\mathbf{1}_n}R_H(\mathbf{y}).
\]

As $H$ is $k$-uniform, we have
\begin{align*}
\sum_{e\in E(H)}\frac{1}{|e|-1}\sum_{\{u,v\}\subseteq e}(y_u-y_v)^2
&=\frac{1}{2(k-1)}\sum_{e\in E(H)}\sum_{u\in e}\sum_{v\in e\setminus\{u\}}(y_u-y_v)^2\\
&=\frac{1}{2(k-1)}\sum_{u\in V(H)}\sum_{v\in V(H)\setminus\{u\}}\sum_{e\in E_{uv}}(y_u-y_v)^2\\
&=\frac{1}{k-1}\sum_{\{u,v\}\subseteq V(H)}e_{uv}(y_u-y_v)^2,
\end{align*}
so the result follows.
\end{proof}

\section{Bounds for the normalized Laplacian eigenvalues} \label{sec3}

To warm up, we give bounds from the largest and the second smallest normalized Laplacian eigenvalues of hypergraphs, extending known results from graphs to hypergraphs (see, e.g. \cite{chung}).

For a hypergraph $H$ and any subset $X\subset V(H)$, let $\overline{X}=V(H)\setminus X$.

\begin{theorem}\label{subset}
Let $H$ be  $k$-uniform hypergraph with $m$ edges. For any subset $\emptyset\ne X\subset V(H)$,
\[
\lambda_1(H)\ge \frac{|\partial X|km}{(k-1)\mathrm{vol}(X)(km-\mathrm{vol}(X))}.
\]
\end{theorem}

\begin{proof}
Let $\mathbf{x}$ be an $n$-dimensional vector defined on $V(H)$ with
\[
x_u=\begin{cases}
\mathrm{vol}(X),& \mbox{if }u\notin X,\\
-\mathrm{vol}(\overline{X}),& \mbox{otherwise.}
\end{cases}
\]
By Lemma \ref{n-},
\begin{align*}
\lambda_1(H)&\ge \frac{\sum_{\{u,v\}\subseteq V(H)}e_{uv}(x_u-x_v)^2}{(k-1)\sum_{u\in V(H)}d_ux_u^2}\\
&=\frac{\sum_{u\in X}\sum_{v\in \overline{X}}e_{uv}(\mathrm{vol}(X)+\mathrm{vol}(\overline{X}))^2}{(k-1)\left(\sum_{u\in X}d_u\mathrm{vol}(\overline{X})^2+\sum_{u\in \overline{X}}d_u\mathrm{vol}(X)^2\right)}\\
&=\frac{|\partial X|(\mathrm{vol}(X)+\mathrm{vol}(\overline{X}))}{(k-1)\mathrm{vol}(X)\mathrm{vol}(\overline{X})}.
\end{align*}
As $H$ is $k$-uniform, we have $\mathrm{vol}(X)+\mathrm{vol}(\overline{X})=km$, so the result follows.
\end{proof}

In the case $k=2$, the result above may be found in \cite[Theorem 3.5]{LGS}.

\begin{corollary}\label{Delta}
Let $H$ be a $k$-uniform hypergraph on with $m$ edges and maximum degree $\Delta$.
Then
$\lambda_1(H)\ge \frac{km}{km-\Delta}$.
\end{corollary}
\begin{proof}
Let $u$ be a vertex of $H$ with maximum degree and $X=\{u\}$. Then the result follows from Theorem \ref{subset} by noting that $|\partial X|=(k-1)\Delta$ and $\mathrm{vol}(X)=\Delta$.
\end{proof}

For a hypergraph $H$, an independent set of $H$ is a subset of $V(H)$, in which any pair of vertices is non-adjacent.
The independence number of $H$, written as $\alpha(H)$,  is the maximum cardinality of all independent sets in $H$.
So, if $X$ is an independent set of a $k$-uniform hypergrph $H$, then $|\partial X|=(k-1)\mathrm{vol}(X)$ and hence the following result is an immediate consequence of Theorem \ref{subset}.

\begin{corollary}\label{inde}
Let $H$ be a $k$-uniform hypergraph with $m$ edges.
For any independent set $X$ of $H$, $\lambda_1(H)\ge \frac{km}{km-\mathrm{vol}(X)}$. Moreover, if the minimum degree of $H$ is $\delta$, then $\lambda_1(H)\ge \frac{km}{km-\alpha\delta}$, where $\alpha$ is the independence number of $H$.
\end{corollary}

\begin{theorem}\label{lambda1}
For a hypergraph $H$ on $n$ vertices, $\lambda_1(H)\le 2$ with equality if and only if $H$ has
a $2$-uniform bipartite component.
%an ordinary bipartite component.
Moreover, if $H$ is $k$-uniform, then $\lambda_1(H)\le \frac{k}{k-1}$ and it is attained if $|E(H)|<n$.
\end{theorem}

\begin{proof}
Let $\mathbf{x}$ be an eigenvector corresponding to $\lambda_1(H)$ and $\mathbf{y}=D(H)^{-1/2}\mathbf{x}$.
Then
\[
\lambda_1(H)=\frac{\mathbf{x}^\top \mathcal{L}(H)\mathbf{x}}{\mathbf{x}^\top \mathbf{x}}=\frac{\sum_{e\in E(H)}\frac{1}{|e|-1}\sum_{\{u,v\}\subseteq e}(y_u-y_v)^2}{\sum_{u\in V(H)}d_uy_u^2}.
\]
Note that for each edge $e\in E(H)$, \[
\sum_{\{u,v\}\subseteq e}(y_u-y_v)^2=\sum_{\{u,v\}\subseteq e}(y_u^2+y_v^2-2y_uy_v).
\]
For $e\in E(H)$, let $y_e=\sum_{u\in e}y_u$. Then
\begin{align*}
\sum_{\{u,v\}\subseteq e}(y_u^2+y_v^2-2y_uy_v)
&=(|e|-1)\sum_{u\in e}y_u^2-2\sum_{\{u,v\}\subseteq e}y_uy_v\\
&=(|e|-1)\sum_{u\in e}y_u^2-\sum_{u\in e}y_u\sum_{v\in e,v\ne u}y_v\\
&=(|e|-1)\sum_{u\in e}y_u^2-\sum_{u\in e}y_u(y_e-y_u)\\
&=|e|\sum_{u\in e}y_u^2-y_e\sum_{u\in e}y_u\\
&=|e|\sum_{u\in e}y_u^2-y_e^2\\
&\le |e|\sum_{u\in e}y_u^2,
\end{align*}
so
\begin{equation}\label{e}
\sum_{e\in E(H)}\frac{1}{|e|-1}\sum_{\{u,v\}\subseteq e}(y_u-y_v)^2\le \sum_{e\in E(H)}\frac{|e|}{|e|-1}\sum_{u\in e}y_u^2.
\end{equation}
As $|e|\ge 2$ for each edge $e\in E(H)$, we have
\begin{align*}
\sum_{e\in E(H)}\frac{|e|}{|e|-1}\sum_{u\in e}y_u^2
&\le 2\sum_{e\in E(H)}\sum_{u\in e}y_u^2\\
&=2\sum_{u\in V(H)}\sum_{e\in E_u}y_u^2\\
&=2\sum_{u\in V(H)}d_uy_u^2,
\end{align*}
so $\lambda_1(H)\le 2$.

Suppose that  $\lambda_1(H)=2$. Then by the above argument, we have
\begin{equation} \label{BP}
\sum_{u\in e}y_u=0
\end{equation}
and
\begin{equation} \label{BP2}
\sum_{u\in e}y_u^2\ne 0 \Rightarrow |e|=2
\end{equation}
for any $e\in E(H)$.

From \eqref{BP} and \eqref{BP2}, $y_u=0$ for any vertex $u$ in some edge containing at least three vertices. So, if a component of $H$ has an edge containing at least three vertices, then $y_u=0$ for any vertex of this component. Note that  $\mathbf{y}$ is nonzero. So there is at least a component $F$ that is
a $2$-uniform hypergraph, that is, an ordinary graph. If $F$ is not bipartite, then it has at least one odd cycle and hence from \eqref{BP},  $y_u=0$ for each vertex $u$ on this cycle, implying that $y_u=0$ for each $u\in V(F)$ by the connectivity of $F$. It thus follows that $F$ is bipartite.

Conversely, suppose that  one component $F$ of $H$ is an ordinary bipartite graph. Then it is easily seen that $\lambda_1(F)=2$. By the above argument, $\lambda_1(H)\le 2$. Evidently,
$\lambda_1(H)$ is the maximum of all the largest normalized Laplacian eigenvalues of the components of $H$. So $\lambda_1(H)=2$

If $H$ is $k$-uniform, then $|e|=k$ for each $e\in E(H)$ and so from \eqref{e} we have
\[
\sum_{e\in E(H)}\frac{1}{|e|-1}\sum_{\{u,v\}\subseteq e}(y_u-y_v)^2\le \frac{k}{k-1}\sum_{e\in E(H)}\sum_{u\in e}y_u^2=\frac{k}{k-1}\sum_{u\in V(H)}d_uy_u^2.
\]
Therefore, $\lambda_1(H)\le \frac{k}{k-1}$ with equality if and only if there is a nonzero vector
$\mathbf{y}$
such that  $\sum_{u\in e}y_u=0$ for any $e\in E(H)$, equivalently,
the homogeneous system of $|E(H)|$ linear equations $\sum_{u\in e}y_u=0$ with  $e\in E(H)$
in $n$ unknowns $y_{u}$ with $u\in V(G)$ has a nonzero solution,
which is definitely guaranteed  if $|E(H)|<n$.
\end{proof}

It was shown in \cite{chung} that for any non-complete graph $G$, $\lambda_{n-1}(G)\le 1$.
The case for hypergraphs is somewhat different.
For example, let $R_{4,3}$ be a $4$-vertex $3$-uniform hypergraph with vertex set $V(R_{4,3})=\{v_1,v_2,v_3,v_4\}$ and edge set  $E(R_{4,3})=\{\{v_1,v_2,v_3\},\{v_1,v_2,v_4\},\{v_1,v_3,v_4\} \}$. This is not the complete
$4$-vertex $3$-uniform hypergraph, though any two vertices are adjacent in $R_{4,3}$.
It is easy to see that
\[
\mathcal{L}(R_{4,3})=\begin{pmatrix}
1&-\frac{1}{\sqrt{6}}&-\frac{1}{\sqrt{6}}&-\frac{1}{\sqrt{6}}\\
-\frac{1}{\sqrt{6}}&1&-\frac{1}{4}&-\frac{1}{4}\\
-\frac{1}{\sqrt{6}}&-\frac{1}{4}&1&-\frac{1}{4}\\
-\frac{1}{\sqrt{6}}&-\frac{1}{4}&-\frac{1}{4}&1
\end{pmatrix}.
\]
By a direct calculation,
$\lambda_3(R_{4,3})=\frac{5}{4}>1$.

\begin{theorem}\label{1}
For an $n$-vertex hypergraph $H$ with at least two non-adjacent vertices, $\lambda_{n-1}(H)\le 1$.
\end{theorem}

\begin{proof}
Denote by $z$ and $w$ the two non-adjacent vertices in $H$.
Let $\mathbf{x}$ be the vector with
\[
x_u=\begin{cases}
d_z,  & \mbox{if }u=w,\\
-d_w,  & \mbox{if }u=z,\\
0,   & \mbox{otherwise.}
\end{cases}
\]
As
\[
\mathbf{x}^\top D(H)\mathbf{1}_n=\sum_{u\in V(H)}x_ud_u=d_wd_z-d_zd_w=0,
\]
we have by Lemma \ref{n-} that
\begin{align*}
\lambda_{n-1}(H)&=\min_{\mathbf{x}\in \mathbb{R}^n\setminus\{\mathbf{0}_n\}\atop
\mathbf{x}\bot D(H)\mathbf{1}_n} R_H(\mathbf{x})\\
&\le \frac{\sum_{e\in E(H)}\frac{1}{|e|-1}\sum_{\{u,v\}\subseteq e}(x_u-x_v)^2}{\sum_{u\in V(H)}d_ux_u^2}.
\end{align*}
Note that
\begin{align*}
&\quad \sum_{e\in E(H)}\frac{1}{|e|-1}\sum_{\{u,v\}\subseteq e}(x_u-x_v)^2\\
&=\sum_{e\in E(H)}\frac{1}{|e|-1}\sum_{\{u,w\}\subseteq e}d_z^2+\sum_{e\in E(H)}\frac{1}{|e|-1}\sum_{\{u,z\}\subseteq e}d_w^2\\
&=d_z^2\sum_{e\in E_w}\sum_{u\in e\setminus\{w\}}\frac{1}{|e|-1}+d_w^2\sum_{e\in E_z}
\sum_{u\in e\setminus \{z\}}\frac{1}{|e|-1}\\
&=d_z^2d_w+d_w^2d_z
\end{align*}
and
\[
\sum_{u\in V(H)}d_ux_u^2=d_wd_z^2+d_zd_w^2.
\]
So
\[
\lambda_{n-1}(H)\le \frac{d_z^2d_w+d_w^2d_z}{d_wd_z^2+d_zd_w^2}=1. \qedhere
\]
\end{proof}

Let $H$ be a connected hypergraph.
The distance between vertices $u$ and $v$ in $H$ is the length (number of edges) of a shortest path connecting them, and the diameter of $H$ is the maximum distance between  any pair of vertices in $H$.

\begin{theorem}
Let $H$ be an $n$-vertex connected $k$-uniform hypergraph with diameter $d$. Then $\lambda_{n-1}(H)>\frac{2}{dk(k-1)^2\mathrm{vol}(V(H))}$.
\end{theorem}

\begin{proof} By Lemma \ref{n-}, there is a vector
$\mathbf{x}$  with
\[
\lambda_{n-1}(H)=\frac{\sum_{e\in E(H)}\sum_{\{u,v\}\subseteq e}(x_u-x_v)^2}{(k-1)\sum_{u\in V(H)}d_ux_u^2}.
\]
Let $w\in V(H)$ with $|x_w|=\max\{|x_u|:u\in V(H)\}$.
Assume that $x_w>0$.
As $\sum_{u\in V(H)}d_ux_u=0$, there exists a vertex $z\in V(H)$ with $x_z<0$.
Let $P$ be the path from $w$ to $z$. Then $|E(P)|\le d$.
For any edge $e\in E(H)$, let $a_{e}=\max\{x_u:u\in e\}$ and $b_{e}=\min\{x_u:u\in e\}$. By the Cauchy-Schwarz inequality,
\begin{align*}
\sum_{\{u,v\}\subseteq e}(x_u-x_v)^2\ge \frac{2}{k(k-1)}\left( \sum_{\{u,v\}\subseteq e}|x_u-x_v|\right)^2\ge \frac{2}{k(k-1)}(a_{e}-b_{e})^2.
\end{align*}
Then
\begin{align*}
\sum_{e\in E(P)}\sum_{\{u,v\}\subseteq e}(x_u-x_v)^2&\ge\frac{2}{k(k-1)}
\sum_{e\in E(P)}(a_{e}-b_{e})^2\\
&\ge \frac{2}{|E(P)|k(k-1)}\left(\sum_{e\in E(P)}|a_{e}-b_{e}| \right)^2\\
&\ge \frac{2}{dk(k-1)}\left(\sum_{e\in E(P)}|a_{e}-b_{e}| \right)^2\\
&\ge \frac{2}{dk(k-1)}\left(x_w-x_z \right)^2\\
&> \frac{2}{dk(k-1)}x_w^2.
\end{align*}
Therefore,
\begin{align*}
\lambda_{n-1}(H)&=\frac{\sum_{e\in E(H)}\sum_{\{u,v\}\subseteq e}(x_u-x_v)^2}{(k-1)\sum_{u\in V(H)}d_ux_u^2}\\
&\ge \frac{\sum_{e\in E(P)}\sum_{\{u,v\}\subseteq e}(x_u-x_v)^2}{(k-1)\sum_{u\in V(H)}d_ux_u^2}\\
&>\frac{2x_w^2}{dk(k-1)^2\sum_{u\in V(H)}d_ux_u^2}\\
&\ge \frac{2x_w^2}{dk(k-1)^2\sum_{u\in V(H)}d_ux_w^2}\\
&=\frac{2}{dk(k-1)^2\mathrm{vol}(V(H))},
\end{align*}
as desired.
\end{proof}

\section{Interlacing theorem} \label{sec4}

Chen et al. \cite{chen}
provided an edge version of Cauchy's interlacing theorem on the normalized Laplacian eigenvalues for an ordinary graph $G$ with $e\in E(G)$,
\[
\lambda_{i+1}(G)\le \lambda_i(G-e)\le\lambda_{i-1}(G) \mbox{ for each $i=1,\dots,n$}
\]
with $\lambda_0(G)=2$ and $\lambda_{n+1}(G)=0$.
Let $K_4^3$ be the complete  $4$-vertex $3$-uniform hypergraph.
For the hypergraph $R_{4,3}$ defined as above, it is obvious that $R_{4,3}=K_4^3-e$ with edge $e=\{v_2,v_3,v_4\}$.
By a direct calculation, $\lambda_1(K_4^3)=\lambda_2(K_4^3)=\lambda_3(K_4^3)=\frac{4}{3}$
and $\lambda_2(R_{4,3})=\lambda_3(R_{4,3})=\frac{5}{4}$. So the above interlacing inequalities do not apply to hypergraphs directly.

Let $H$ be an $n$-vertex hypergraph which may contain isolated vertices. Let $W$ be the set of isolated vertices in $H$.
Let  $M=\mbox{diag}\{m(u):u\in V(H)\}$, where
\[
m(u)=\begin{cases}
1, & \mbox{if }u\in W,\\
d_H(u), & \mbox{otherwise}.
\end{cases}
\]
Let  $i=1,\dots,n-1$.
For vectors $\mathbf{x}$ and $\mathbf{x}^{(j)}\in \mathbb{R}^n$ with $j=1,\dots,n-i$, let
\[
\mathbf{y}=M^{-1/2}\mathbf{x}\mbox{ and } \mathbf{y}^{(j)}=M^{1/2}\mathbf{x}^{(j)}.
\]
Then
\[
\mathbf{y}^\top\mathbf{y}^{(j)}=(M^{-1/2}\mathbf{x})^\top (M^{1/2}\mathbf{x}^{(j)})=\mathbf{x}^\top M^{-1/2}M^{1/2}\mathbf{x}^{(j)}=\mathbf{x}^\top\mathbf{x}^{(j)},
\]
so  $\mathbf{x}\bot\mathbf{x}^{(j)}$ if and only if $\mathbf{y}\bot \mathbf{y}^{(j)}$.
Note that $M^{1/2}D(H)^{-1/2}$ is a diagnol matrix, where the entry corresponding to $u$ is $0$ if $u\in W$ and $1$ otherwise.
As the entry in row and column corresponding to each isolated vertex  in $L(H)$ is zero, we have
\begin{align} \label{add}
\frac{\mathbf{x}^\top \mathcal{L}(H)\mathbf{x}}{\mathbf{x}^\top \mathbf{x}}&=\frac{\mathbf{y}^\top M^{1/2} D(H)^{-1/2}L(H)D^{-1/2}M^{1/2}\mathbf{y}}{(M^{1/2}\mathbf{y})^\top (M^{1/2}\mathbf{y})} \notag\\
&=\frac{\sum_{e\in E(H)}\frac{1}{|e|-1}\sum_{\{u,v\}\subseteq e}(y_u-y_v)^2}{\sum_{u\in \overline{W}}d_uy_u^2+\sum_{u\in W}y_u^2}.
\end{align}
Denote by $M_H(\mathbf{y})$ the expression in the right side of \eqref{add}.
Particularly, if $W=\emptyset$, then $M_H(\mathbf{y})=R_H(\mathbf{y})$.
The well known Courant-Fischer theorem (see Theorem 4.2.11 in \cite{HJ}) may be stated as
\begin{align*}
\lambda_i(H)&=\max_{\mathbf{x}^{(1)},\dots,\mathbf{x}^{(n-i)}\in\mathbb{R}^n}
\min_{\mathbf{x}\in \mathbb{R}^n\setminus\{\mathbf{0}_n\} \atop \mathbf{x}\bot\mathbf{x}^{(1)},\dots,\mathbf{x}^{(n-i)}}\frac{\mathbf{x}^\top \mathcal{L}(H)\mathbf{x}}{\mathbf{x}^\top\mathbf{x}}\\
&=\max_{\mathbf{x}^{(1)},\dots,\mathbf{x}^{(n-i)}\in\mathbb{R}^n}
\min_{\mathbf{x}\in \mathbb{R}^n\setminus\{\mathbf{0}_n\} \atop \mathbf{y}\bot\mathbf{y}^{(1)},\dots,\mathbf{y}^{(n-i)}}M_H(\mathbf{y})\\
&=\max_{\mathbf{y}^{{(1)}},\dots,\mathbf{y}^{(n-i)}\in\mathbb{R}^n}
\min_{\mathbf{y}\in \mathbb{R}^n\setminus\{\mathbf{0}_n\} \atop\mathbf{y}\bot\mathbf{y}^{(1)},\dots,\mathbf{y}^{(n-i)}}M_H(\mathbf{y}).
\end{align*}
Similarly,
\begin{align*}
\lambda_i(H)%&=\min_{\mathbf{x}^{(1)},\dots,\mathbf{x}^{(i-1)}\in\mathbb{R}^n}
%\max_{\mathbf{x}\in \mathbb{R}^n\setminus\{\mathbf{0}_n\}\atop
%\mathbf{x}\bot\mathbf{x}^{(1)},\dots,\mathbf{x}^{(i-1)}}\frac{\mathbf{x}^\top \mathcal{L}(H)\mathbf{x}}{\mathbf{x}^\top\mathbf{x}}\\
=\min_{\mathbf{y}^{(1)},\dots,\mathbf{y}^{(i-1)}\in \mathbb{R}^n}\max_{\mathbf{y}\in \mathbb{R}^n\setminus\{\mathbf{0}_n\}\atop \mathbf{y}\bot \mathbf{y}^{(1)},\dots,\mathbf{y}^{(i-1)}}M_H(\mathbf{y}).
\end{align*}

For $i=1,\dots,n$,
$\mathbf{e}_i$ denotes the $n$-dimensional column vector with its $i$-th entry $e_i$ to be $1$ and any other entry $0$. For a hypergraph $H$ with $e\in H$, as usual, $H-e$ denotes the subhypergraph obtained from $H$ by removing the edge $e$. Here we note that $H-e$ (and even $H$) may have isolated vertices in the following theorem.

\begin{theorem}\label{interlacing}
Let $H$ be an $n$-vertex $k$-uniform hypergraph and $e\in E(H)$. Then each $i=1, \dots, n$,
\[
\lambda_{i+k-1}(H)\le \lambda_i(H-e)\le \lambda_{i-k+1}(H)
\]
with $\lambda_j(H)=\frac{k}{k-1}$ if $j< 1$ and $\lambda_{j}(H)=0$ if $j>n$.
\end{theorem}

\begin{proof} By Theorems \ref{n0} and \ref{lambda1}, $\lambda_n(H)=0$ and $\lambda_1(H)\le \frac{k}{k-1}$.

If $H$ contains exactly one edge $e$, the result is trivial.
Suppose that $H$ contains at least two edges.	
Let $e=\{v_1,\dots,v_k\}$ with $d_{H}(v_1)\ge 2$.
Let
\[
W=\{v: v\in e \mbox{ and } v\not\in e' \mbox{ for any } e'\in E(H)\setminus \{e\}\}.
\]
That is, $W$ is the set of isolated vertices in $H':=H-e$ and $t=|W|$.

\noindent
{\bf Claim 1.}
 $\lambda_{i-k+1}(H)\ge\lambda_i(H')$  for $i=k, \dots, n$.

It is trivial for $i=n$. Suppose that $i=k, \dots, n-1$. By the Courant-Fischer theorem,
\begin{align*}
\lambda_i(H')&=\max_{\mathbf{x}^{(1)},\dots,\mathbf{x}^{(n-i)}\in\mathbb{R}^n}
\min_{\mathbf{z}\in \mathbb{R}^n\setminus\{\mathbf{0}_n\} \atop \mathbf{z}\bot\mathbf{x}^{(1)},\dots,\mathbf{x}^{(n-i)}}M_{H'}(\mathbf{z})\\
%&\le \max_{\mathbf{x}^{{(1)}},\dots,\mathbf{x}^{(n-i)}\in\mathbb{R}^n}\min_{\mathbf{z}\in \mathbb{R}^n\setminus\{\mathbf{0}_n\} \atop {\mathbf{z}\bot \mathbf{x}^{{(j)}}, j=1,\dots,n-i\atop z_1=-(k-1)z_\ell,\ell=2,\dots,k}}M_{H'}(\mathbf{z})\\
&\le  \max_{\mathbf{x}^{{(1)}},\dots,\mathbf{x}^{(n-i)}\in\mathbb{R}^n}
\min_{\mathbf{z}\in \mathbb{R}^n\setminus\{\mathbf{0}_n\} \atop {\mathbf{z}\bot \mathbf{x}^{{(j)}}, j=1,\dots,n-i\atop \mathbf{z}\bot  (\mathbf{e}_1+(k-1)\mathbf{e}_\ell),\ell=2,\dots,k}}
M_{H'}(\mathbf{z})
\end{align*}
For fixed $\mathbf{x}^{(1)},\dots,\mathbf{x}^{(n-i)}\in \mathbb{R}^n$, as $(n-i)+(k-1)<n$, there is a vector $\mathbf{z}\in \mathbb{R}^n\setminus\{\mathbf{0}_n\}$ satisfying
$\mathbf{z}\bot \mathbf{x}^{{(j)}}$ for $j=1,\dots,n-i$ and $\mathbf{z}\bot  (\mathbf{e}_1+(k-1)\mathbf{e}_\ell)$ for $\ell=2,\dots,k$.
Let $a=z_{v_k}$. Then
$z_1=-(k-1)a$ and $z_{v_\ell}=a$ for $\ell=2,\dots,k$.
Thus
\begin{align*}
M_{H'}(\mathbf{z})&=\frac{\sum_{f\in E(H)}\frac{1}{k-1}\sum_{\{u,v\}\subseteq f}(z_u-z_v)^2-\frac{1}{k-1}\sum_{\{u,v\}\subseteq e}(z_u-z_v)^2}{\sum_{u\in V(H)}d_{H}(u)z_u^2-\sum_{u\in e\setminus W}z_u^2}\\
&=\frac{\sum_{f\in E(H)}\frac{1}{k-1}\sum_{\{u,v\}\subseteq f}(z_u-z_v)^2-k^2a^2}{\sum_{u\in V(H)}d_{H}(u)z_u^2-(k(k-1)-t)a^2}\\
&\le \frac{\sum_{f\in E(H)}\frac{1}{k-1}\sum_{\{u,v\}\subseteq f}(z_u-z_v)^2-k^2a^2}{\sum_{u\in V(H)}d_{H}(u)z_u^2-k(k-1)a^2}.
\end{align*}
As in \eqref{e}, we have $\sum_{f\in E(H)}\frac{1}{k-1}\sum_{\{u,v\}\subseteq f}(z_u-z_v)^2\le \frac{k}{k-1}\sum_{u\in V(H)}d_H(u)z_u^2$, so
\[
M_{H'}(\mathbf{z})
\le\frac{\sum_{f\in E(H)}\frac{1}{k-1}\sum_{\{u,v\}\subseteq f}(z_u-z_v)^2}{\sum_{u\in V(H)}d_{H}(u)z_u^2}\\
=M_H(\mathbf{z}).
\]
Thus
\[
\lambda_i(H')\le  \max_{\mathbf{x}^{(1)},\dots,\mathbf{x}^{(n-i)}\in\mathbb{R}^n}
\min_{\mathbf{z}\in \mathbb{R}^n\setminus\{\mathbf{0}_n\} \atop {\mathbf{z}\bot \mathbf{x}^{{(j)}}, j=1,\dots,n-i\atop \mathbf{z}\bot  (\mathbf{e}_1+(k-1)\mathbf{e}_\ell),\ell=2,\dots,k}}M_H(\mathbf{z}).
\]
Now by the Courant-Fischer theorem again, we have
\begin{align*}
\lambda_{i-k+1}(H)&=\max_{\mathbf{x}^{(1)},\dots,\mathbf{x}^{(n-i+k-1)}\in \mathbb{R}^n}
\min_{\mathbf{z}\in \mathbb{R}^n\setminus\{\mathbf{0}_n\}
\atop \mathbf{z}\bot \mathbf{x}^{(j)}, j=1,\dots,n-i+k-1}M_H(\mathbf{z})\\
&\ge \max_{\mathbf{x}^{(1)},\dots,\mathbf{x}^{(n-i+k-1)}\in \mathbb{R}^n
\atop {\mathbf{x}^{(n-i+\ell-1)}=\mathbf{e}_1+(k-1)\mathbf{e}_\ell,\ell=2,\dots,k}}
\min_{\mathbf{z}\in \mathbb{R}^n\setminus\{\mathbf{0}_n\} \atop \mathbf{z}\bot \mathbf{x}^{(j)}, j=1,\dots,n-i+k-1}M_H(\mathbf{z})\\
&= \max_{\mathbf{x}^{(1)},\dots,\mathbf{x}^{(n-i)}\in \mathbb{R}^n}
\min_{\mathbf{z}\in \mathbb{R}^n\setminus\{\mathbf{0}_n\} \atop {\mathbf{z}\bot \mathbf{x}^{(j)}, j=1,\dots,n-i\atop
\mathbf{z}\bot
(\mathbf{e}_1+(k-1)\mathbf{e}_\ell),\ell=2,\dots,k
}}M_H(\mathbf{z}).
\end{align*}
So Claim 1 follows.

\noindent
{\bf Claim 2.}
$\lambda_{i}(H')\ge \lambda_{i+k-1}(H)$  for $i=1, \dots, n-k+1$.

By the Courant-Fischer theorem, we have
\begin{align*}
\lambda_i(H')&=\min_{\mathbf{x}^{(1)},\dots,\mathbf{x}^{(i-1)}\in\mathbb{R}^n}
\max_{\mathbf{x}\in \mathbb{R}^n\setminus\{\mathbf{0}_n\}\atop \mathbf{x}\bot \mathbf{x}^{(j)},j=1,\dots,i-1}M_{H'}(\mathbf{x})\\
&\ge\min_{\mathbf{x}^{{(1)}},\dots,\mathbf{x}^{(i-1)}\in\mathbb{R}^n}
\max_{\mathbf{x}\in \mathbb{R}^n\setminus\{\mathbf{0}_n\} \atop {\mathbf{x}\bot \mathbf{x}^{(j)},j=1,\dots,i-1,\atop x_1=\dots =x_k}}\frac{\sum_{f\in E(H)}\frac{1}{k-1}\sum_{\{u,v\}\subseteq f}(x_u-x_v)^2}{\sum_{u\in V(H)}d_{H}(u)x_u^2-(k-t)x_1^2}\\
&=\min_{\mathbf{x}^{(1)},\dots,\mathbf{x}^{(i-1)}\in\mathbb{R}^n}
\max_{\mathbf{x}\in \mathbb{R}^n\setminus\{\mathbf{0}_n\}\atop {\mathbf{x}\bot \mathbf{x}^{(j)},j=1,\dots,i-1 \atop \mathbf{x}\bot (\mathbf{e}_1-\mathbf{e}_\ell),\ell=2,\dots,k}}\frac{\sum_{f\in E(H)}\frac{1}{k-1}\sum_{\{u,v\}\subseteq f}(x_u-x_v)^2}{\sum_{u\in V(H)}d_{H}(u)x_u^2-(k-t)x_1^2}\\
&\ge \min_{\mathbf{x}^{(1)},\dots,\mathbf{x}^{(i-1)}\in\mathbb{R}^n}
\max_{\mathbf{x}\in \mathbb{R}^n\setminus\{\mathbf{0}_n\} \atop {\mathbf{x}\bot \mathbf{x}^{(j)},j=1,\dots,i-1\atop \mathbf{x}\bot (\mathbf{e}_1-\mathbf{e}_\ell),\ell=2,\dots,k}}M_H(\mathbf{x})\\
&\ge \min_{\mathbf{x}^{(1)},\dots,\mathbf{x}^{(i+k-2)}\in\mathbb{R}^n}
\max_{\mathbf{x}\in \mathbb{R}^n\setminus\{\mathbf{0}_n\} \atop
\mathbf{x}\bot \mathbf{x}^{(j)},j=1,\dots,i+k-2}M_H(\mathbf{x})\\
&=\lambda_{i+k-1}(H),
\end{align*}
so Claim 2 follows.

Now the result follows by combining Claims 1 and 2.
\end{proof}

\section{Cheeger inequality} \label{sec5}

In this section, we give a Cheeger inequality to relate the Cheeger constant to the normalized Laplacian eigenvalues of hypergraphs.
We extend such the Cheeger inequality
established by Chung \cite{chung} (see also \cite{chung4th,Ch}) from graphs to uniform hypergraphs as follows.

\begin{theorem}\label{cheeup}
Let $H$ be a connected $k$-uniform hypergraph on $n$ vertices.
Then
\[
\frac{h^2(H)}{2(k-1)^2}\le 1-\sqrt{1-\frac{h^2(H)}{(k-1)^2}}\le \lambda_{n-1}(H)\le \frac{2}{k-1}h(H).
 \]
\end{theorem}

Let $H$ be a connected $k$-uniform hypergraph on $n$ vertices. We divide the proof of Theorem \ref{cheeup} into  the following three lemmas by merging the techniques from \cite{Ch,Mo}.  
Obviously,  $\frac{h^2(H)}{2(k-1)^2}\le 1-\sqrt{1-\frac{h^2(H)}{(k-1)^2}}$.
Losing a factor $k-1$ in estimating $\sum_{\{u,v\}\subseteq V(H)} e_{uv}(y_u+y_v)^2$ in the proof of Lemma \ref{ad1} led to  the incorrect form:
\[
\max\left\{\frac{h^2(H)}{2(k-1)},1-\sqrt{1-\frac{h^2(H)}{(k-1)^2}}\right\}\le \lambda_{n-1}(H).
\]

\begin{lemma} $\lambda_{n-1}(H)\le \frac{2}{k-1}h(H)$.
\end{lemma}

\begin{proof}
Let $S\subset V(H)$ be the set of vertices such that $h(H)=\frac{|\partial S|}{\mathrm{vol}(S)}$.
Let $\mathbf{x}$ be an $n$-dimensional column vector defined on $V(H)$ with
\[
x_u=\begin{cases}
\frac{1}{\mathrm{vol}(S)},& \mbox{if }u\in S,\\
-\frac{1}{\mathrm{vol}(\overline{S})},& \mbox{otherwise}.
\end{cases}
\]
As
\[
\mathbf{x}^\top D(H)\mathbf{1}_n=\sum_{u\in V(H)}x_ud_u=\sum_{u\in S}\frac{d_u}{\mathrm{vol}(S)}-\sum_{u\in \overline{S}}\frac{d_u}{\mathrm{vol}(\overline{S})}=0,
\]
we have by Lemma \ref{n-} that
\begin{align*}
\lambda_{n-1}(H)&\le \frac{\sum_{\{u,v\}\subseteq V(H)}e_{uv}(x_u-x_v)^2}{(k-1)\sum_{u\in V(H)}d_ux_u^2}\\
&=\frac{\sum_{u\in S}\sum_{v\in \overline{S}}e_{uv}\left(\frac{1}{\mathrm{vol}(S)}+\frac{1}{\mathrm{vol}(\overline{S})}\right)^2}{\left(k-1\right)\left(\sum_{u\in S}d_u\frac{1}{\mathrm{vol}(S)^2}+\sum_{u\in \overline{S}}d_u\frac{1}{\mathrm{vol}(\overline{S})^2}\right)}\\
&=\frac{1}{k-1}\left(\frac{1}{\mathrm{vol}(S)}+\frac{1}{\mathrm{vol}(\overline{S})}\right)\sum_{u\in S}\sum_{v\in \overline{S}}e_{uv}.
\end{align*}
As $\mathrm{vol}(S)\le \mathrm{vol}(\overline{S})$, we have
\[
\lambda_{n-1}(H)\le \frac{2}{k-1}\cdot\frac{1}{\mathrm{vol}(S)}\sum_{u\in S}\sum_{v\in \overline{S}}e_{uv}
=\frac{2}{k-1}\cdot\frac{|\partial S|}{\mathrm{vol}(S)}
=\frac{2}{k-1}h(H). \qedhere
\]
\end{proof}

\begin{lemma} \label{ad1}
$\lambda_{n-1}(H)\ge \frac{h^2(H)}{2(k-1)^2}$.
\end{lemma}

In the published version, it appeared to be $\lambda_{n-1}(H)\ge \frac{h^2(H)}{2(k-1)}$.

\begin{proof}
Let $\mathbf{z}$ be an eigenvector of $\mathcal{L}(H)$ corresponding to $\lambda_{n-1}(H)$. Let $\mathbf{x}=D(H)^{-1/2}\mathbf{z}$.
From $\lambda_{n-1}(H) \mathbf{z}=\mathcal{L}(H)\mathbf{z}$, we have
\[
\lambda_{n-1}(H) z_u=z_u-\sum_{v\in V(H)\setminus\{u\}}\frac{e_{uv}}{(k-1)\sqrt{d_ud_v}}z_v \mbox{ for } u\in V(H),
\]
i.e.,
\[
\lambda_{n-1}(H) \sqrt{d_u}x_u=\sqrt{d_u}x_u-\sum_{v\in V(H)\setminus\{u\}}\frac{e_{uv}}{(k-1)\sqrt{d_ud_v}}\sqrt{d_v}x_v  \mbox{ for } u\in V(H).
\]
Let $W=\{u\in V(H):x_u>0\}$. it is evident that $\emptyset\ne W\subset V(H)$.
Assume that $\mathrm{vol}(W)\le \mathrm{vol}(\overline{W})$.
%$\sum_{u\in W}d_u\le \sum_{u\notin W}d_u$.
Let $\mathbf{y}$ be the $n$-dimensional vector defined on $V(H)$ with
\[
y_u=\begin{cases}
	x_u, & \mbox{if }u\in W,\\
	0, & \mbox{otherwise.}
\end{cases}
\]
Then
\begin{align*}
\lambda_{n-1}(H)\sum_{u\in W}d_ux_u^2
&=\sum_{u\in W}\sqrt{d_u}x_u\left(\lambda_{n-1}(H) \sqrt{d_u}x_u\right)\\
&=\sum_{u\in W}\sqrt{d_u}x_u\left(\sqrt{d_u}x_u-\sum_{v\ne u}\frac{e_{uv}}{(k-1)\sqrt{d_ud_v}}\sqrt{d_v}x_v \right)\\
&=\sum_{u\in W}x_u\left(d_ux_u-\sum_{v\in V(H)\setminus \{u\}}\frac{e_{uv}}{k-1}x_v \right).
\end{align*}
As $d_u=\frac{1}{k-1}\sum_{v\in V(H)\setminus \{u\}}e_{uv}$ for any $u\in V(H)$, we have
\begin{align*}
&\quad \sum_{u\in W}x_u\left(d_ux_u-\sum_{v\in V(H)\setminus \{u\}}\frac{e_{uv}}{k-1}x_v \right)\\
&=\frac{1}{k-1}\sum_{u\in W}x_u\sum_{v\in V(H)\setminus \{u\}}e_{uv}(x_u-x_v)\\
&=\frac{1}{k-1}\left(\sum_{u\in W}x_u\sum_{v\in W\setminus\{u\}}e_{uv}(x_u-x_v)+\sum_{u\in W}x_u\sum_{v\in \overline{W}}e_{uv}(x_u-x_v)\right)\\
&\ge \frac{1}{k-1}\left(\sum_{u\in W}x_u\sum_{v\in W\setminus \{u\}}e_{uv}(x_u-x_v)+\sum_{u\in W}x_u\sum_{v\in \overline{W}}e_{uv}x_u\right)\\
&=\frac{1}{k-1}\left(\sum_{u\in W}y_u\sum_{v\in W\setminus\{u\}}e_{uv}(y_u-y_v)+\sum_{u\in W}y_u\sum_{v\in \overline{W}}e_{uv}y_u\right).
\end{align*}
Note that in the sum $\sum_{u\in W}y_u\sum_{v\in W\setminus\{u\}}e_{uv}(y_u-y_v)$, any pair $\{u,v\}\subseteq W$ contributes $e_{uv}(y_u-y_v)y_u+e_{uv}(y_v-y_u)y_v=e_{uv}(y_u-y_v)^2$. That is,
\[
\sum_{u\in W}y_u\sum_{v\in W\setminus\{u\}}e_{uv}(y_u-y_v)=\sum_{\{u,v\}\subseteq W}e_{uv}(y_u-y_v)^2.
\]
Note also that
\[
\sum_{u\in W}y_u\sum_{v\in \overline{W}}e_{uv}y_u=\sum_{u\in W}\sum_{v\in \overline{W}}e_{uv}(y_u-y_v)^2
\]
and
\[
\sum_{\{u,v\}\subseteq V(H)}e_{uv}(y_u-y_v)^2=\sum_{\{u,v\}\subseteq W}e_{uv}(y_u-y_v)^2+\sum_{u\in W}\sum_{v\in \overline{W}}e_{uv}(y_u-y_v)^2.
\]
Thus
\begin{equation}\label{M}
\lambda_{n-1}(H) \sum_{u\in W}d_ux_u^2\ge\frac{1}{k-1}\sum_{\{u,v\}\subseteq V(H)} e_{uv}(y_u-y_v)^2,
\end{equation}
which implies that
\begin{equation}\label{3a}
\begin{aligned}
\lambda_{n-1}(H)
& \ge \frac{\sum_{\{u,v\}\subseteq V(H)} e_{uv}(y_u-y_v)^2}{(k-1)\sum_{u\in W}d_uy_u^2}\\
&=\frac{\sum_{\{u,v\}\subseteq V(H)} e_{uv}(y_u-y_v)^2\sum_{\{u,v\}\subseteq V(H)} e_{uv}(y_u+y_v)^2}{(k-1)\sum_{u\in W}d_uy_u^2\sum_{\{u,v\}\subseteq V(H)} e_{uv}(y_u+y_v)^2}\\
&\ge \frac{\left(\sum_{\{u,v\}\subseteq V(H)} e_{uv}|y_u^2-y_v^2|\right)^2}{(k-1)\sum_{u\in W}d_uy_u^2\sum_{\{u,v\}\subseteq V(H)} e_{uv}(y_u+y_v)^2}\\
&\ge \frac{\left(\sum_{\{u,v\}\subseteq V(H)} e_{uv}|y_u^2-y_v^2|\right)^2}{2(k-1)^2\left(\sum_{u\in W}d_uy_u^2\right)^2},
\end{aligned}
\end{equation}
where the second inequality follows from the Cauchy-Schwarz inequality and the third inequality follows because
\[
\sum_{\{u,v\}\subseteq V(H)} e_{uv}(y_u+y_v)^2
\le \sum_{u\in V(H)}\sum_{v\in V(H)\setminus\{u\}}e_{uv}(y_u^2+y_v^2)
=2(k-1)\sum_{u\in V(H)}d_uy_u^2.
\]
The factor $k-1$ is lost  in $2(k-1)\sum_{u\in V(H)}d_uy_u^2$ in early versions.

Assume that $\{y_u:u\in V(H) \}=\{t_0,t_1,\dots,t_m\}$, where $0=t_0<t_1<\dots<t_m$.
Let $V_i=\{u\in V(H):y_u\ge t_i\}$ for $i=0,1,\dots,m$.
Then
%\begin{equation}\label{3b}
\begin{align*}
\sum_{\{u,v\}\subseteq V(H)} e_{uv}|y_u^2-y_v^2|
&=\sum_{0\le a<c\le m}\sum_{u\in V(H)\atop y_u=t_a}\sum_{v\in V(H)\atop  y_v=t_c}e_{uv}(t_c^2-t_a^2)\\
&=\sum_{0\le a<c\le m}\sum_{u\in V(H)\atop  y_u=t_a}\sum_{v\in V(H)\atop  y_v=t_c}e_{uv}\sum_{i=a+1}^{c}(t_i^2-t_{i-1}^2)\\
&=\sum_{i=1}^m\sum_{a=0}^{i-1}\sum_{c=i}^m\sum_{u\in V(H)\atop y_u=t_a}\sum_{v\in V(H)\atop y_v=t_c}e_{uv}(t_i^2-t_{i-1}^2)\\
&=\sum_{i=1}^m\sum_{u\in \overline{V_i}}\sum_{v\in V_i}e_{uv}(t_i^2-t_{i-1}^2)\\
&=\sum_{i=1}^m|\partial V_i|\left(t_i^2-t_{i-1}^2\right).
\end{align*}
%\end{equation}
Note that $V_m\subset V_{m-1}\subset \dots V_1=W\subseteq V_0=V(H)$. So $\mathrm{vol}(V_i)\le \mathrm{vol}(W)\le \mathrm{vol}(\overline{W})\le \mathrm{vol}(\overline{V_i})$, and thus  $|\partial V_i|\ge h(H)\mathrm{vol}(V_i)$ for $i=1,\dots,m$. Therefore
\begin{equation}\label{3b}
\begin{aligned}
\sum_{\{u,v\}\subseteq V(H)} e_{uv}|y_u^2-y_v^2|
&\ge h(H)\sum_{i=1}^m \mathrm{vol}(V_i)\left(t_i^2-t_{i-1}^2\right)\\
&=h(H)\sum_{i=1}^m \mathrm{vol}(V_i)t_i^2-\sum_{i=1}^{m-1}\mathrm{vol}(V_{i+1})t_i^2\\
&=h(H)\sum_{i=1}^{m-1} \left(\mathrm{vol}(V_i\setminus V_{i+1})t_i^2+ \mathrm{vol}(V_m)t_m^2\right)\\
&=h(H)\sum_{u\in W}d_uy_u^2.
\end{aligned}
\end{equation}

By combining \eqref{3a} and \eqref{3b}, we complete the proof.
\end{proof}

\begin{lemma}
$\lambda_{n-1}(H)\ge 1-\sqrt{1-\frac{h^2(H)}{(k-1)^2}}$.
\end{lemma}

\begin{proof} We use notations from the proof of the previous lemma.
From  \eqref{M} and $\sum_{u\in W}d_ux_u^2=\sum_{u\in W}d_uy_u^2$, we have
\[
\lambda_{n-1}(H)\ge \frac{\sum_{\{u,v\}\subseteq V(H)}e_{uv}(y_u-y_v)^2}{(k-1)\sum_{u\in W}d_uy_u^2}:=M.
\]
As
\begin{align*}
\sum_{\{u,v\}\subseteq V(H)} e_{uv}(y_u+y_v)^2&=\sum_{\{u,v\}\subseteq V(H)}e_{uv}\left(2(y_u^2+y_v^2)-(y_u-y_v)^2\right)\\
&=2\sum_{\{u,v\}\subseteq V(H)}e_{uv}(y_u^2+y_v^2)-M(k-1)\sum_{u\in W}d_uy_u^2\\
&=(2-M)(k-1)\sum_{u\in W}d_uy_u^2,
\end{align*}
we have from \eqref{3a} and \eqref{3b} that
\begin{align*}
M
\ge \frac{h^2(H)\left(\sum_{u\in W}d_uy_u^2\right)^2}{(2-M)(k-1)^2\left(\sum_{u\in W}d_uy_u^2\right)^2}=\frac{h^2(H)}{(2-M)(k-1)^2},
\end{align*}
i.e.,
 \[
(k-1)^2M^2-2(k-1)^2M+h^2(H)\le 0,
\]
so
\[
\lambda_{n-1}(H)\ge M\ge1-\sqrt{1-\frac{h^2(H)}{(k-1)^2}}\,. \qedhere
\]
\end{proof}

\section{Discrepancy inequality} \label{sec6}

For any two nonempty subsets $X,Y$ of hypergraph $H$, let
$M(X,Y)$ be the multiset of edges in $H$ with at least one vertex in $X$ and at least one different vertex in $Y$. The multiplicity of an edge $e\in M(X,Y)$ is $\sum_{u\in e\cap X}|e\cap Y\setminus\{u\}|$.
Particulaly, $\partial X=M(X,V(H)\setminus X)$ if $\emptyset\ne X\subset V(H)$.
Let $m(X,Y)=|M(X,Y)|$.
Then $m(X,Y)=\sum_{u\in X}\sum_{v\in Y}e_{uv}$.
If $H$ is an ordinary graph, it is of interest to find a fixed $\alpha$ (depending up normalized Laplacian spectrum) so that
\[
\left|m(X,Y)-\frac{\mathrm{vol}(X)\mathrm{vol}(Y)}{\mathrm{vol}(V(H))}\right|\le \alpha \frac{\sqrt{\mathrm{vol}(X)\mathrm{vol}(\overline{X})\mathrm{vol}(Y)\mathrm{vol}(\overline{Y})}}{\mathrm{vol}(V(H))}
\]
for all $\emptyset\ne X,Y\subset V(H)$. Such an inequality that bounds the edge distribution between two sets based on the spectral gap is known as a discrepancy inequality \cite{Ch}.
It states that a small spectral gap of a graph implies that the edge distribution is close to random.

We extend Chung's discrepancy inequality \cite[Theorem 3]{chung2}
 (or \cite[Theorem 5.2]{chung}) from graphs to $k$-uniform hypergraphs.

\begin{theorem}\label{dis}
Let $H$ be a $k$-uniform hypergraph on $n$ vertices.
For any two subsets $X,Y\subseteq V(H)$, we have
\[
\left|\frac{m(X,Y)}{k-1}-\frac{\mathrm{vol}(X)\mathrm{vol}(Y)}{\mathrm{vol}(V(H))}\right|\le \lambda \frac{\sqrt{\mathrm{vol}(X)\mathrm{vol}(\overline{X})\mathrm{vol}(Y)\mathrm{vol}(\overline{Y})}}{\mathrm{vol}(V(H))},
\]
where $\lambda=\max\{|1-\lambda_i(H)|: i=1,\dots,n-1 \}$.
\end{theorem}
\begin{proof}
Let $\mathbf{z}_X$ and $\mathbf{z}_Y$ be the $n$-dimensonal column vectors defined on $V(H)$ with
\[
\mathbf{z}_X(u)=\begin{cases}
1, & \mbox{if }u\in X\\
0, & \mbox{otherwise}
\end{cases}  \mbox{ and }
\mathbf{z}_Y(u)=\begin{cases}
1, & \mbox{if }u\in Y,\\
0, & \mbox{otherwise}.
\end{cases}
\]
Note that
\[
\mathbf{z}_X^\top A(H)\mathbf{z}_Y=\sum_{u\in X}(A\mathbf{z}_Y)_{u}=\sum_{u\in X}\sum_{v\in Y}\frac{e_{uv}}{k-1}=\frac{m(X,Y)}{k-1}
\]
and
\begin{align*}
\mathbf{z}_X^\top A(H)\mathbf{z}_Y&=\mathbf{z}_X^\top D(H)^{1/2} (I_n-\mathcal{L}(H))D(H)^{1/2}\mathbf{z}_Y\\
&=\left(D(H)^{1/2}\mathbf{z}_X\right)^\top \left(I_n-\mathcal{L}(H)\right)\left(D(H)^{1/2}\mathbf{z}_Y\right).
\end{align*}
Let $\mathbf{x}_1,\dots,\mathbf{x}_n$ be the orthonormal eigenvectors of $\mathcal{L}(H)$ with $\mathcal{L}(H)\mathbf{x}_i=\lambda_i(H)\mathbf{x}_i$ for $i=1,\dots,n$, where we may choose $\mathbf{x}_n=\frac{1}{\sqrt{\mathrm{vol}(V(H))}}D(H)^{1/2}\mathbf{1}_n$ by Theorem \ref{n0}.
Assume that $D(H)^{1/2}\mathbf{z}_X=\sum_{i=1}^na_i\mathbf{x}_i$ and $D(H)^{1/2}\mathbf{z}_Y=\sum_{i=1}^nb_i\mathbf{x}_i$ for $i=1,
\dots, n$.
Then
\begin{align*}
\mathbf{z}_X^\top A(H)\mathbf{z}_Y&=\left(\sum_{i=1}^na_i\mathbf{x}_i\right)^\top (I_n-\mathcal{L}(H))\left( \sum_{i=1}^nb_i\mathbf{x}_i\right)\\
&=\sum_{i=1}^na_ib_i(1-\lambda_i(H))\\
&=a_nb_n+\sum_{i=1}^{n-1}a_ib_i(1-\lambda_i(H)),
\end{align*}
so
\[
\left|\frac{m(X,Y)}{k-1}-a_nb_n \right|=\left|\sum_{i=1}^{n-1}a_ib_i(1-\lambda_i(H)) \right|.
\]
Note that
\[
a_n=\mathbf{x}_n^\top D(H)^{1/2}\mathbf{z}_X=\frac{\mathrm{vol}(X)}{\sqrt{\mathrm{vol}(V(H))}}, \
b_n=\mathbf{x}_n^\top D(H)^{1/2}\mathbf{z}_Y=\frac{\mathrm{vol}(Y)}{\sqrt{\mathrm{vol}(V(H))}},
\]
and
\[
\sum_{i=1}^{n-1}a_i^2=\frac{\mathrm{vol}(X)\mathrm{vol}(\overline{X})}{\mathrm{vol}(V(H))}, \ \sum_{i=1}^{n-1}b_i^2=\frac{\mathrm{vol}(Y)\mathrm{vol}(\overline{Y})}{\mathrm{vol}(V(H))}.
\]
So, by the Cauchy-Schwarz inequality,
\begin{align*}
\left|\frac{m(X,Y)}{k-1}-\frac{\mathrm{vol}(X)\mathrm{vol}(Y)}{\mathrm{vol}(V(H))}\right|&=\left|\sum_{i=1}^{n-1}a_ib_i(1-\lambda_i(H))\right|\\
&\le \lambda\left| \sum_{i=1}^{n-1}a_ib_i\right|\\
&\le \lambda\sqrt{\sum_{i=1}^{n-1}a_i^2\sum_{i=1}^{n-1}b_i^2}\\
&=\lambda \frac{\sqrt{\mathrm{vol}(X)\mathrm{vol}(\overline{X})\mathrm{vol}(Y)\mathrm{vol}(\overline{Y})}}{\mathrm{vol}(V(H))},
\end{align*}
as desired.
\end{proof}

%Particularly, if $k=2$, then $H$ is a ordinary graph and Theorem \ref{dis} is the discrepancy inequality on graphs, see .

Recall that $\alpha (H)$ denotes the independence number of a hypergraph $H$.

\begin{corollary}
Let $H$ be a $k$-uniform hypergrpah on $n$ vertices and $m$ edges with minimum degree $\delta$. Then
\[
\alpha(H)\le \frac{km\lambda }{(1+\lambda)\delta},
\]
where $\lambda=\max\{|1-\lambda_i(H)|:i=1,\dots,n-1 \}$.
\end{corollary}
\begin{proof}
Let $X=Y$ be a maximum independent set of $H$.
Then $|X|=\alpha(H)$ and $m(X,Y)=0$.
By Theorem \ref{dis}, we have
\[
\frac{\mathrm{vol}(X)^2}{\mathrm{vol}(V(H))}\le \lambda \frac{\mathrm{vol}(X)\mathrm{vol}(\overline{X})}{\mathrm{vol}(V(H))},
\]
so
$\mathrm{vol}(X)\le \lambda \mathrm{vol}(\overline{X})$.
Note that $\mathrm{vol}(\overline{X})=km-\mathrm{vol}(X)$.
So
\[
\mathrm{vol}(X)\le \frac{\lambda}{1+\lambda}km.
\]
As $\mathrm{vol}(X)\ge \delta|X|\ge \delta\alpha(H)$, we have
\[
\alpha(H)\le \frac{km\lambda }{(1+\lambda)\delta}. \qedhere
\]
\end{proof}

%\bigskip
%
%\noindent {\bf Acknowledgements}
%The authors  would like to thank the reviewers for constructive comments and suggestions.
%%This work was supported by the National Natural Science Foundation of China (No.~12071158).
%
%\bigskip

\noindent
{\bf Funding}  This work was supported by the National Natural Science Foundation of China (No.~12071158).

%\bigskip
%
%\noindent
%{\bf Data availability statement}
%Data sharing not applicable to this article as no datasets were generated or
%analysed during the current study.
%
%\bigskip
%
%\noindent
%{\bf Declarations}
%
%\bigskip
%
%\noindent
%{\bf Conflict of interest} %The authors declare that they have no conflict of interest.
%The authors have no relevant financial or non-financial interests to disclose.
%

\end{document}